\documentclass[final,12pt]{article}

\usepackage[T2A]{fontenc}
\usepackage[cp1251]{inputenc}
\usepackage[english]{babel}
\usepackage{amsfonts,amsmath,amsthm,amscd,amssymb,latexsym,amsbsy,pb-diagram,cite,caption2}
\usepackage[dvips]{graphicx}
\usepackage[mathscr]{eucal}

\usepackage{showkeys}

\newtheorem{definition}{Definition}[section]
\newtheorem{theorem}[definition]{Theorem}
\newtheorem{lemma}[definition]{Lemma}
\newtheorem{proposition}[definition]{Proposition}
\newtheorem{corollary}[definition]{Corollary}

\theoremstyle{definition}
\newtheorem{remark}[definition]{Remark}
\newtheorem{example}[definition]{Example}

\numberwithin{equation}{section}

\DeclareMathOperator{\diam}{diam}

\begin{document}
\begin{center}
{\Large\bf Extended by Balk metrics}
\end{center}

\bigskip
\begin{center}
{\bf O. Dovgoshey and D. Dordovskyi}
\end{center}

\begin{abstract}
Let $X$ be a nonempty set and $\mathcal{F}(X)$ be the set of
nonempty finite subsets of $X$. The paper deals with the extended
metrics $\tau:\mathcal{F}(X)\to\mathbb{R}$ recently introduced by
Peter Balk. Balk's metrics and their restriction to the family of
sets $A$ with $|A|\leqslant n$ make possible to consider "distance
functions" with $n$ variables and related them quantities. In
particular, we study such type generalized diameters
$\diam_{\tau^n}$ and find conditions under which
$B\mapsto\diam_{\tau^n}B$ is a Balk's metric. We prove the necessary
and sufficient conditions under which the restriction $\tau$ to the
set of $A\in\mathcal{F}(X)$ with $|A|\leqslant 3$ is a symmetric
$G$-metric. An infinitesimal analog for extended by Balk metrics is
constructed.
\end{abstract}

\bigskip



\section{Introduction}

The following generalized metrics were introduced by P.~Balk in 2009
for applications to some inverse geophysical problems.

Let $X$ be a nonempty set and $\mathcal{F}(X)$ be the set of all
nonempty finite subsets of $X$.

\begin{definition}\cite{Ba}\label{def1.1}
A function $\tau:\mathcal{F}(X)\to\mathbb{R}$ is an extended (by
Balk) metric on $X$  if the equivalence
\begin{equation}\label{eq1.1}
(\tau(A)=0)\Leftrightarrow(|A|=1),
\end{equation}
and the equality
\begin{equation}\label{eq1.2}
\tau(A\cup B)\leqslant\tau(A\cup C)+\tau(C\cup B)
\end{equation}
hold for all $A,B,C\in\mathcal{F}(X)$.
\end{definition}

\noindent{\bf Example 1.1} \cite{Ba}. If $\rho$ is a metric on $X$,
then the function $\tau(A)=\diam_{\rho}(A)$, with
$\diam_{\rho}(A)=\sup\{\rho(x,y) : x,y\in A\}$, is an extended by
Balk metric.

If $\tau$ is an extended by Balk metric on $X$ then, as shown in
Proposition \ref{pr2.1}, the function $\tau^2:X^2\to\mathbb{R}$,
with
\begin{equation}\label{eq1.4}
\tau^2(x,y)=\tau(Im(x,y)),\quad
Im(x,y)=\begin{cases} \{x\}\qquad\text{ if } x=y\\
\{x,y\}\quad\text{ if } x\neq y,
\end{cases}
\end{equation}
is a metric on $X$. Analogously, for all integer numbers $k\geqslant
1$ we can define the functions $\tau^k:X^k\to\mathbb{R}$ as
\begin{equation}\label{eq1.5}
\tau^k(x_1,\ldots,x_k)=\tau(Im(x_1,\ldots,x_k)),
\end{equation}
where $Im(x_1,\ldots,x_k)$ is the image of the set  $\{1,\ldots,k\}$
under the map $i\mapsto x_i$,
\begin{equation}\label{eq1.6}
(x\in Im(x_1,\ldots,x_k))\Leftrightarrow(\exists\,
i\in\{1,\ldots,k\}:x = x_i).
\end{equation}
Formula \eqref{eq1.5} turns to formula \eqref{eq1.4} when $k=2$,
thus we obtain a "generalized metric" which is a function of $k$
variables (while the usual metric is a function of two variables).

In what follows the important role will play some "generalized
diameters" generated by $\tau^k$.

\begin{definition}\label{def1.2*}
Let $X\neq\varnothing$, $k$ be an integer positive number and let
$\tau:\mathcal{F}(X)\to\mathbb{R}$ be an extended by Balk metric.
For every nonempty $A\subseteq X$ we set
$$
\diam_{\tau^k}\,A=\sup\{\tau^k(x_1,\ldots,x_k)\,:\,x_1,\ldots,x_k\in
A\}
$$
that is equivalent to
\begin{equation}\label{n*}
\diam_{\tau^k}\,A=\sup\{\tau(B)\,:\,B\subseteq A,\,|B|\leqslant k\}.
\end{equation}
\end{definition}

\begin{remark}
It is clear that $\diam_{\tau^k}\,A$ is the usual diameter of $A$ if
$k=2$. Definition \ref{def1.1} implies $\diam_{\tau^1}\,A=0$ for
every $A\in\mathcal{F}(X)$.
\end{remark}

In Theorem \ref{t2.11} of the second section of the paper we obtain
a structural characteristic of extended by Balk metrics
$\tau:\mathcal{F}(X)\to\mathbb{R}$ for which $\tau(A) =
\diam_{\tau^k}A$ holds with all $A\in\mathcal{F}(X)$ and $k\geqslant
2$.

In the third section we study the relationship between $\tau^3$ and
the so-called $G$-metrics which were introduced by Zead Mustafa and
Brailey Sims in 2006.

\begin{definition}\cite{5*}\label{def1.3}
Let $X$ be a nonempty set. A function $G:X^3\to\mathbb{R}$ is called
a $G$-metric if the following properties hold.
\begin{itemize}
\item[$(i)$] $G(x,y,z)=0$ for $x=y=z$.
\item[$(ii)$] $0<G(x,x,y)$ for $x\neq y$.
\item[$(iii)$] $G(x,x,y)\leqslant G(x,y,z)$ for $z\neq y$.
\item[$(iv)$] $G(x_1,x_2,x_3)=G(x_{\sigma_1}, x_{\sigma_2},
x_{\sigma_3})$ for every permutation $\sigma$ of the set $\{1,2,3\}$
and every $(x_1,x_2,x_3)\in X^3$.
\item[$(v)$] $G(x,y,z)\leqslant G(x,a,a)+G(a,y,z)$ for all $a,x,y,z\in
X$.
\end{itemize}
\end{definition}

\begin{definition}\label{def1.4}
A $G$-metric is called symmetric if the equality $G(x,y,y)=G(y,x,x)$
holds for all $x,y\in X$.
\end{definition}

\begin{remark}\label{r1.4*}
In \cite{5*} $G$-metrics were defined as some functions $G$ with the
codomain $[0,\infty)$, which is slightly different from Definition
\ref{def1.3}. In this connection it should be pointed out that
conditions $(i)-(iv)$ of Definition \ref{def1.3} imply the
nonnegativity of $G$. Indeed, it is sufficient to prove $G(y,x,x)>0$
for $x\neq y$, that follows from $0<G(x,x,y)=G(x,y,x)=G(y,x,x)$.
\end{remark}

We shall prove that for every symmetric $G$-metric on $X$ there is
an increasing extended by Balk metric
$\tau:\mathcal{F}(X)\to\mathbb{R}$ such that $\tau^3=G$. Conversely,
an arbitrary $\tau^3$ is a $G$-metric if the corresponding extended
by Balk metric $\tau:\mathcal{F}(X)\to\mathbb{R}$ is increasing.
(See Theorem \ref{t2.9*}).

The infinitesimal structure of spaces $(X,\tau)$ with extended by
Balk metrics $\tau$ is investigated in the fourth section. In
particular, we transfer the extended by Balk metrics $\tau$ from $X$
to spaces which are pretangent to $(X, \tau^2)$. The pretangent
spaces to the general metric spaces were introduced in \cite{DM1}
(see also \cite{DM2}). For convenience, we recall some related
definitions.

Let $(X,d)$ be a metric space and let $p\in X$. Fix a sequence
$\tilde{r}$ of positive real numbers $r_n$ which tend to zero. The
sequence $\tilde{r}$ will be called a {\it normalizing sequence}.
Let us denote by $\tilde{X}_p$ the set of all sequences of points
from $X$ which tend to $p$.

\begin{definition}\label{def3.1}
Two sequences $\tilde{x},\tilde{y}\in\tilde{X}_p$,
$\tilde{x}=(x_n)_{n\in\mathbb{N}}$ and
$\tilde{y}=(y_n)_{n\in\mathbb{N}}$ are mutually stable with respect
to a normalizing sequence $\tilde{r}=(r_n)_{n\in\mathbb{N}}$, if
there is a finite limit
\begin{equation}\label{eq3.1}
\underset{n\rightarrow\infty}{\lim}\frac{d(x_n,y_n)}{r_n}:=
\tilde{d}_{\tilde{r}}(\tilde{x},\tilde{y})=
\tilde{d}(\tilde{x},\tilde{y}).
\end{equation}
\end{definition}

The family $\tilde{F}\subseteq\tilde{X}_p$ is {\it self-stable} with
respect to $\tilde{r}$, if every two
$\tilde{x},\tilde{y}\in\tilde{F}$  are mutually stable, $\tilde{F}$
is {\it maximal self-stable} if $\tilde{F}$ is self-stable and for
an arbitrary $\tilde{z}\in\tilde{X}_p\setminus\tilde{F}$ there is
$\tilde{x}\in\tilde{F}$ such that $\tilde{x}$ and $\tilde{z}$ are
not mutually stable. Zorn's lemma leads to the following

\begin{proposition}\label{pr3.2}
Let $(X,d)$ be a metric space and let $p\in X$. Then for every
normalizing sequence $\tilde{r}=\{r_n\}_{n\in\mathbb{N}}$ there
exists a maximal self-stable family $\tilde{X}_{p,\tilde{r}}$ such
that $\tilde{p}=\{p,p,...\}\in\tilde{X}_{p,\tilde{r}}$.
\end{proposition}

Let us consider a function
$\tilde{d}:\tilde{X}_{p,\tilde{r}}\times\tilde{X}_{p,\tilde{r}}\to\mathbb{R}$,
where
$\tilde{d}(\tilde{x},\tilde{y})=\tilde{d}_{\tilde{r}}(\tilde{x},\tilde{y})$
is defined by \eqref{eq3.1}. Obviously, $\tilde{d}$ is symmetric and
nonnegative. Moreover, the triangle inequality for $d$ implies
$$
\tilde{d}(\tilde{x},\tilde{y})\leq
\tilde{d}(\tilde{x},\tilde{z})+\tilde{d}(\tilde{z},\tilde{y})
$$
for all $\tilde{x},\tilde{y},\tilde{z}\in\tilde{X}_{p,\tilde{r}}$.
Hence $(\tilde{X}_{p,\tilde{r}},\tilde{d})$ is a pseudometric space.

Define a relation $\sim$ on $\tilde{X}_{p,\tilde{r}}$ by
$\tilde{x}\sim\tilde{y}$ if and only if
$\tilde{d}_{\tilde{r}}(\tilde{x},\tilde{y})=0$. Let us denote by
$\Omega_{p,\tilde{r}}^X$  the set of equivalence classes in
$\tilde{X}_{p,\tilde{r}}$  under the equivalence relation $\sim$.
For $\alpha,\beta\in\Omega_{p,\tilde{r}}^X$ set
\begin{equation}\label{eq3.2}
\rho(\alpha,\beta)=\tilde{d}(\tilde{x},\tilde{y}),
\end{equation}
where $\tilde{x}\in\alpha$ and $\tilde{y}\in\beta$, then $\rho$ is a
metric on $\Omega_{p,\tilde{r}}^X$ (see, for example, \cite[Ch.~4,
Theorem~15]{Kelly}).

\begin{definition}\label{def3.3}
The space $(\Omega_{p,\tilde{r}}^X,\rho)$ is pretangent to the space
$X$ at the point $p$ with respect to a normalizing sequence
$\tilde{r}$.
\end{definition}

Let $\tau:\mathcal{F}(X)\to\mathbb{R}$ be an extended by Balk
metric, let $p\in X$ and let $(\Omega^X_{p,\tilde{r}},\rho)$ be a
pretangent space to the metric space $(X, \tau^2)$. Now the
"lifting" of $\tau$ on $(\Omega^X_{p,\tilde{r}},\rho)$ is defined as
follows. Let $\mathfrak{U}$ be a nontrivial ultrafilter on
$\mathbb{N}$. For
$\{\alpha_1,\ldots,\alpha_n\}\in\mathcal{F}(\Omega^X_{p,\tilde{r}})$,
$(x^1_m)_{m\in\mathbb{N}}\in\alpha_1,\ldots,(x^n_m)_{m\in\mathbb{N}}\in\alpha_n$
set
$$
\mathcal{X}_{\tau}(\{\alpha_1,\ldots,\alpha_n\})=\mathfrak{U}-
\lim\frac{\tau(Im(x^1_m,\ldots,x^n_m))}{r_m}.
$$
In Theorem \ref{t3.5} it is proved that $\mathcal{X}_{\tau}$ is an
extended by Balk metric on $(\Omega^X_{p,\tilde{r}},\rho)$ and
$\mathcal{X}^2_{\tau}=\rho$. Theorem \ref{t3.10} provides a
characteristic of extended metrics
$\tau:\mathcal{F}(X)\to\mathbb{R}$ for which the equality
$$
\mathcal{X}_{\tau}(A)=\diam_{\rho}\,A
$$
holds for every $A\in\mathcal{F}(\Omega^X_{p,\tilde{r}})$. This
result is used in Corollary \ref{c3.14} for characterization of
$\tau$ for which $\mathcal{X}_{\tau}$ are the extended
"ultrametrics", i.e. satisfy the inequality
$$
\mathcal{X}_{\tau}(A\cup B)\leqslant \max\{\mathcal{X}_{\tau}(A\cup
C), \mathcal{X}_{\tau}(B\cup C)\}
$$
instead of inequality \eqref{eq1.2}.

\section{Extended by Balk metrics and generalized diameters}

Let $X$ be a nonempty set and $\tau:\mathcal{F}(X)\to\mathbb{R}$ be
an extended by Balk metric on $X$. Set

\begin{equation}\label{eq2.1}
\tau^{2}(x,y) := \begin{cases} \tau(\{x,y\}),\quad \text{ if } x\neq y\\
0,\qquad \text{ if } $x=y$
\end{cases}
\end{equation}
for every ordered pair $(x,y)\in X\times X$, where $\{x,y\}$ is the
set whose elements are the points $x$ and $y$.

\begin{proposition}\label{pr2.1}
The function $\tau^{2} : X^2 \to\mathbb{R}$ is a metric for every
nonempty set $X$ and extended metric
$\tau:\mathcal{F}(X)\to\mathbb{R}$.
\end{proposition}

\begin{proof}
Obviously, the function $\tau^{2}$ is symmetric and by \eqref{eq1.1}
$\tau^{2}(x,y)=0$ if and only if $x=y$. Putting in \eqref{eq1.2}
$A=B=C$ we obtain $\tau(A)\leqslant 2\tau(A)$ for every
$A\in\mathcal{F}(X)$ that is an equivalent to $\tau(A)\geqslant 0$.
The last inequality implies the nonnegativity of the function
$\tau^{2}$. It remains to prove the triangle inequality for
$\tau^2$. Let $x,y,z$ be arbitrary points from $X$. Putting
$A=\{x\}$, $B=\{y\}$ and $C=\{z\}$ into inequality \eqref{eq1.2} we
obtain

$$
\begin{gathered}
\tau^{2}(x,y) = \tau(\{x\}\cup\{y\}) \leqslant \tau(\{x\}\cup\{z\})
+ \tau(\{z\}\cup\{y\})\\ \leqslant \tau(\{x,z\}) + \tau(\{z,y\}) =
\tau^{2}(x,z) + \tau^{2}(z,y).
\end{gathered}
$$
Thus the triangle inequality is satisfied.
\end{proof}

If $d$ is a metric and $\tau$ is an extended by Balk metric on the
same set $X$ and the equality $d(x,y)=\tau^2(x,y)$ holds for all
$x,y\in X$, we say that $\tau$ is {\it compatible} with $d$.

\begin{remark}\label{r2.2}
The nonnegativity of $\tau$ was earlier proved in \cite{Ba}.
\end{remark}

Recall that a mapping $f:X\to Y$ from a partially ordered set
$(X,\leqslant_X)$ to a partially ordered set $(Y,\leqslant_Y)$ is
called {\it increasing} if the implication
$$
(x\leqslant_X y)\Rightarrow (f(x)\leqslant_Y f(y))
$$
holds for all $x,y\in X$.

Let us put in order the set $\mathcal{F}(X)$ by the set-theoretic
inclusion $\subseteq$ and consider $\mathbb{R}$ with the standard
order $\leqslant$. If $\rho$ is a metric on $X$, then the mapping
$$
\mathcal{F}(X)\ni A \mapsto \diam_{\rho}(A)\in\mathbb{R}
$$
is increasing.

\begin{definition}\label{def2.4}
Let $X\neq\varnothing$ and $k$ be an integer number greater or equal
two. A mapping $f:\mathcal{F}(X)\to\mathbb{R}$ is called
$k$-increasing if the implication
$$
(B\subseteq A) \Rightarrow (f(B)\leqslant f(A))
$$
holds for $A,B\in\mathcal{F}(X)$ with $|B|\leqslant k$.
\end{definition}

\begin{remark}\label{r2.3}
It is clear that every increasing mapping
$f:\mathcal{F}(X)\to\mathbb{R}$ is $k$-increasing for every
$k\geqslant 2$. It is not hard to check that, if $|X|\leqslant k+1$,
then all $k$-increasing mappings are increasing.
\end{remark}

The next example shows that for $|X|\geqslant k+2$ there are
extended by Balk metrics on $X$ which are $k$-increasing but not
$k+1$-increasing mappings.

\begin{example}\label{ex2.5}
Let $|X|\geqslant k+2$ and $t_i$, $i=2,\ldots,k+2$ be some numbers
from the interval $(1,2)$ such that $t_k<t_{k+2}<t_{k+1}$ and
$t_i<t_{i+1}$ for $i=2,\ldots,k$. For $A\in\mathcal{F}(X)$ set
\begin{equation}\label{eq2.2}
\tau(A)=
\begin{cases}
0,    \,\,\qquad\text{ for } |A|=1 \\
t_n,  \qquad\text{ for } |A|=n , \text{ if } 2\leqslant n\leqslant k+1\\
t_{k+2}, \,\quad\text{ for } |A|\geqslant k+2.
\end{cases}
\end{equation}
It follows directly from \eqref{eq2.2} and the restrictions to the
numbers $t_n$ that $\tau$ is $k$-increasing but not
$k+1$-increasing. If $|A\cup C|\neq 1\neq |B\cup C|$ and $\tau(A\cup
B)=t_i$, $\tau(A\cup C) = t_j$, $\tau(B\cup C)=t_l$ then
$t_i,t_j,t_l\in(1,2)$. Hence $t_i\leqslant t_j+t_l$ that implies
\eqref{eq1.2}. Assuming, for example, the equality $1=|B\cup C|$ we
obtain the existence of $x\in X$ such that $B=C=\{x\}$. Then
inequality \eqref{eq1.2}turns into an equality. Case $|A\cup C|=1$
is similar.
\end{example}

\begin{lemma}\label{l2.6}
The following conditions are equivalent for all $X\neq\varnothing$,
$\tau:\mathcal{F}(X)\to\mathbb{R}$ and integer numbers $k\geqslant
2$.
\begin{itemize}
\item[$(i)$] The mapping $\tau$ is a $k$-increasing  function from
$(\mathcal{F}(X),\subseteq)$ to $(\mathbb{R},\leqslant)$.
\item[$(ii)$] The inequality
\begin{equation}\label{eq2.3}
\tau(A)\geqslant\max\{\tau(B)\,:\,B\subseteq A, |B|\leqslant k\}
\end{equation}
holds for every $A\in\mathcal{F}(X)$.
\end{itemize}
\end{lemma}

The proof can be obtained directly from definitions and we omit it
here.

\begin{corollary}\label{c2.7}
Let $X\neq\varnothing$ and $k$ be an integer number greater or equal
two. An extended metric $\tau:\mathcal{F}(X)\to\mathbb{R}$ is a
$k$-increasing mapping from $(\mathcal{F}(X),\subseteq)$ to
$(\mathbb{R},\leqslant)$ if and only if the inequality
$\tau(A)\geqslant\diam_{\tau^{k}}A$ holds for every
$A\in\mathcal{F}(X)$ where $\diam_{\tau^{k}}A$ is defined by
relation \eqref{n*}.
\end{corollary}

Let $(X,\leqslant_X)$ and $(Y,\leqslant_Y)$ be partially ordered
sets. A mapping $f:X\to Y$ is called {\it decreasing} if the
implication  $(z\leqslant_X y) \Rightarrow (f(z)\geqslant_Y f(y))$
holds for all $z,y\in X$.

In the following definition the relation $B\subset A$ means that we
have $B\subseteq A$ and $B\neq A$.

\begin{definition}\label{def2.10}
Let $X\neq\varnothing$ and $k\geqslant 2$ be an integer number. A
mapping $f:\mathcal{F}(X)\to\mathbb{R}$ is $k$-weakly decreasing if
for every  $A\in\mathcal{F}(X)$ with $|A|> k$ there is a finite
nonempty set $B\subset A$ such that  $f(B)\geqslant f(A)$.
\end{definition}

\begin{lemma}\label{l2.11}
The following conditions are equivalent for all $X\neq\varnothing$,
$k\geqslant 2$ and mappings $\tau:\mathcal{F}(X)\to\mathbb{R}$.
\begin{itemize}
\item[$(i)$] The mapping $\tau$ is $k$-weakly decreasing.

\item[$(ii)$] The inequality
\begin{equation}\label{eq2.5}
\tau(A)\leqslant\max\{\tau(B)\,:\,B\subseteq A, |B|\leqslant k\}
\end{equation}
holds for every $A\in\mathcal{F}(X)$.
\end{itemize}
\end{lemma}

\begin{proof}
The implication $(ii)\Rightarrow(i)$ follows directly from
Definition \ref{def2.10}. Let us check the implication
$(i)\Rightarrow(ii)$. Assume that condition $(i)$ is true. Let us
prove inequality \eqref{eq2.5} using induction by $|A|$. If
$|A|=1,\ldots,k$ inequality \eqref{eq2.5} is obvious. Suppose that
\eqref{eq2.5} is proved for $|A|\leqslant n$, $n\in\mathbb{N}$.
Assume $|A|=n+1\geqslant k+1$. By $(i)$ the mapping $\tau$ is
$k$-weakly decreasing. Therefore there is $B\subset A$ such that
$\tau(A)\leqslant\tau(B)$. From the inclusion $B\subset A$ follows
the inequality $|B|\leqslant n$. Using the induction hypothesis we
get

\begin{equation}\label{eq2.6}
\tau(A)\leqslant\tau(B)\leqslant\max\{\tau(C)\,:\,C\subseteq B,\,
|C|\leqslant k\}.
\end{equation}
Since $(C\subseteq B)$ implies $(C\subseteq A)$, we obtain
$$
\max\{\tau(C)\,:\,C\subseteq B,\, |C|\leqslant k\} \leqslant
\max\{\tau(C)\,:\,C\subseteq A,\, |C|\leqslant k\}.
$$
The last inequality and \eqref{eq2.6} give \eqref{eq2.5}.
\end{proof}

The next corollary directly follows from Definition \ref{def1.2*}
and Lemma \ref{l2.11}.

\begin{corollary}\label{c2.12}
Let $X\neq\varnothing$ and $k\geqslant 2$ be an integer number. An
extended by Balk metric $\tau:\mathcal{F}(X)\to\mathbb{R}$ is a
$k$-weakly decreasing mapping from $(\mathcal{F}(X),\subseteq)$ into
$(\mathbb{R},\leqslant)$ if and only if the inequality
$\tau(A)\leqslant\diam_{\tau^{k}}(A)$ holds for every
$A\in\mathcal{F}(X)$.
\end{corollary}

Lemmas \ref{l2.6} and \ref{l2.11} give the following.

\begin{corollary}\label{c2.10*}
Let $X\neq\varnothing$, let $k\geqslant 2$ be an integer number and
let $\tau:\mathcal{F}(X)\to\mathbb{R}$ be a $k$-weakly decreasing
mapping. Then $\tau$ is increasing if and only if it is
$k$-increasing.
\end{corollary}
\begin{proof}
It is sufficient to  verify that if $\tau$ is $k$-increasing then
$\tau$ is increasing. Indeed, if $\tau$ is $k$-increasing, then
inequalities \eqref{eq2.3} and \eqref{eq2.5} imply
$$
\tau(A)=\max\{\tau(B)\,:\,B\subseteq A, |B|\leqslant k\},\quad
A\in\mathcal{F}(A).
$$
The increase of $\tau$ follows.
\end{proof}

Combining corollaries \ref{c2.7}, \ref{c2.12} and \ref{c2.10*} we
get

\begin{theorem}\label{t2.11}
The following statements are equivalent for all nonempty $X$,
integer $k\geqslant 2$ and extended metrics
$\tau:\mathcal{F}(X)\to\mathbb{R}$.
\begin{itemize}
\item[$(i)$] The equality $\tau(A) = \diam_{\tau^k}\,A$ holds for
every $A\in\mathcal{F}(X)$, where $\diam_{\tau^k}\,A$ is determined
by Definition \ref{def1.2*}.

\item[$(ii)$] $\tau$ is $k$-increasing and $k$-weakly decreasing.

\item[$(iii)$] $\tau$ is increasing and $k$-weakly decreasing.
\end{itemize}
\end{theorem}

\begin{definition}\label{def1.2}
Let $\rho$ be a metric on $X$ and $\tau:\mathcal{F}(X)\to\mathbb{R}$
be an extended by Balk metric on $X$. We say that $\tau$ is
generated by $\rho$ if $\tau(A)=\diam_{\rho}\,A$ for any
$A\in\mathcal{F}(A)$.
\end{definition}

\begin{theorem}\label{t2.13}
Let $\tau:\mathcal{F}(X)\to\mathbb{R}$ be an extended by Balk metric
on a nonempty set $X$. The following statements are equivalent.
\begin{itemize}
\item[$(i)$] There is a mapping $\mu:X^2\to\mathbb{R}$ such that
$\mathcal{F}(A)=\max\{\mu(x,y)\,:\,x,y\in A\}$ for every
$A\in\mathcal{F}(X)$.
\item[$(ii)$] There is a metric on $X$ which generates $\tau$.
\item[$(iii)$]  $\tau$ is generated by $\tau^2$.
\item[$(iv)$]  $\tau$ is 2-increasing and 2-weakly decreasing.
\item[$(v)$]  $\tau$ is increasing and 2-weakly decreasing.
\end{itemize}
\end{theorem}

\begin{proof}
The implications $(iii)\Rightarrow(ii)$ and $(ii)\Rightarrow(i)$ are
obvious. The equivalences $(v)\Leftrightarrow(iv)$ and
$(iv)\Leftrightarrow(iii)$ follow from Theorem \ref{t2.11}. It
remains to note that $(i)\Leftrightarrow(v)$ follows immediately
from the definitions of increasing mapping and 2-weakly decreasing
one.
\end{proof}

In the next section we will prove an analog of Theorem \ref{t2.13}
for the symmetric $G$-metrics.

\section{Extended by Balk metrics and $G$-metrics}
The domain of $\tau^3$ (see formula \eqref{eq1.5}) is the set
$X^3=X\times X\times X$. Different generalized metrics with this
domain were considered at least since 60s of the last century
\cite{1*,2*,3*}. The so-called $G$-metric (see Definition
\ref{def1.3}) is among the most important from these
generalizations. The $G$-metric was introduced by Mustafa and Sims
\cite{4*,5*} and has applications in the fixed point theory.

In the current section we, in particular, show that the functions
$\tau^3:X^3\to\mathbb{R}$ generated by increasing extended by Balk
metrics $\tau:\mathcal{F}(X)\to\mathbb{R}$ are symmetric (in the
sense of Definition \ref{def1.4}) $G$-metrics on $X$.

\begin{lemma}\label{l2.2*}
Let $X\neq\varnothing$ and let $\tau:\mathcal{F}(X)\to\mathbb{R}$ be
an increasing extended by Balk metric. Then $\tau^3$ is a symmetric
$G$-metric on $X$.
\end{lemma}

\begin{proof}
By Definition \eqref{eq1.4} $\tau^3$ is a symmetric $G$-metric if
and only if the equality $\tau^3(x,y,y)=\tau^3(y,x,x)$ holds for all
$x,y\in X$. This equality immediately follows from \eqref{eq1.5} and \eqref{eq1.6}.

Let us check conditions $(i)-(v)$ of Definition \ref{def1.3}.

\begin{itemize}
\item[$(i)$] For every $x$ the equality $\tau^3(x,x,x)=0$ follows from
\eqref{eq1.1}.

\item[$(ii)$] The inequality $\tau^3(x,x,y)>0$ for $x\neq y$ follows from
the equality $\tau^3(x,x,y)=\tau^2(x,y)$ and the fact that $\tau^2$
is a metric on $X$ (see Proposition \ref{pr2.1}).

\item[$(iii)$] The inequality $\tau^3(x,x,y)\leqslant\tau^3(x,y,z)$
follows because $\tau$ is increasing.

\item[$(iv)$] The arguments of the function $\tau$ on the right-hand side
of equality \eqref{eq1.5} are sets, that automatically gives the
invariance of $\tau^3$ with respect to the permutations of
arguments.

\item[$(v)$] We must prove the inequality
\begin{equation}\label{eq2.2*}
\tau^3(x,y,z)\leqslant\tau^3(x,a,a)+\tau^3(a,y,z)
\end{equation}
for all $x,y,z,a\in X$. The inequality holds if $x=y=z$ since
$\tau^3$ is a nonnegative function and $\tau(x,x,x)=0$. Now let
$x\neq y\neq z\neq x$. Substituting $A=\{x\}$, $B=\{y,z\}$,
$C=\{a\}$ in \eqref{eq1.2} we obtain
$$
\tau^3(x,y,z)=\tau(A\cup B)\leqslant \tau(A\cup C) + \tau(B\cup
C)=\tau^3(x,a,a) + \tau^3(a,y,z).
$$

If $y=z$, inequality \eqref{eq2.2*} is equivalent to the triangle
inequality
$$
\tau^2(x,y)\leqslant\tau^2(x,a)+\tau^2(a,y),
$$
that was proved in Proposition \ref{pr2.1}. Let $x=z$. Then
\eqref{eq2.2*} can be written as
\begin{equation}\label{eq2.3*}
\tau^3(x,x,y)\leqslant\tau^3(x,a,a)+\tau^3(a,y,x).
\end{equation}
Since $\tau$ is increasing, the inequality
$\tau^3(a,y,x)\geqslant\tau^3(a,y,y)$ holds. Therefore, it is
sufficient to check
$\tau^3(x,x,y)\leqslant\tau^3(x,a,a)+\tau^3(a,y,y)$ which again
reduces to the triangle inequality for $\tau^2$. It remains to
consider the case where $x=y$. With this assumption \eqref{eq2.2*}
we get
\begin{equation}\label{eq2.4**}
\tau^3(x,x,z)\leqslant\tau^3(x,a,a)+\tau^3(a,x,z).
\end{equation}
Again from the increase of $\tau$ we obtain
$\tau^3(a,x,z)\geqslant\tau^3(a,z,z)$. Hence it suffices to prove
the inequality $\tau^3(x,x,z)\leqslant\tau^3(x,a,a)+\tau^3(a,z,z)$,
which also follows from the triangle inequality.
\end{itemize}
\end{proof}

Now we want to prove the converse of Lemma \ref{l2.2*}. To do this
it suffices for given symmetric $G$-metric $\Phi:X^3\to\mathbb{R}$
to construct an increasing extended by Balk metric
$\tau:\mathcal{F}(X)\to\mathbb{R}$ such that
\begin{equation}\label{eq2.4*}
\Phi(x_1,x_2,x_3) = \tau(Im(x_1,x_2,x_3)),\quad (x_1,x_2,x_3)\in X^3,
\end{equation}
where $Im(x_1,x_2,x_3)$ was defined
by relation \eqref{eq1.6}. We will carry out this construction in
two steps.

\begin{itemize}
\item[$\bullet$] For given $G$-metric $\Phi:X^3\to\mathbb{R}$ we first find
an increasing mapping $\tilde{\tau}:\mathcal{F}^3(X)\to\mathbb{R}$,
$\mathcal{F}^3(X)=\{A\in\mathcal{F}(X)\,:\,|A|\leqslant 3\}$, which
satisfies \eqref{eq2.4*} and \eqref{eq1.1}, \eqref{eq1.2} with
$\tau=\tilde{\tau}$. (This is almost what we need but the domain of
$\tilde{\tau}$ is $\mathcal{F}^3(X)$).

\item[$\bullet$] Second, we expand $\tilde{\tau}$ to an increasing extended by Balk
metric $\tau:\mathcal{F}(X)\to\mathbb{R}$.
\end{itemize}

\begin{lemma}\label{l2.3*}
Let $X\neq\varnothing$. The following statements are equivalent for
every function $G:X^3\to\mathbb{R}$.
\begin{itemize}
\item[$(i)$] $G$ satisfies condition $(iv)$ of Definition
\ref{def1.3} and is symmetric in the sense that
\begin{equation}\label{eq2.5*}
G(x,y,y)=G(y,x,x)
\end{equation}
for all $x,y\in X$.

\item[$(ii)$] There is a mapping $\tilde{\tau}:\mathcal{F}^3(X)\to\mathbb{R}$
such that equality \eqref{eq2.4*} holds for every $(x_1,x_2,x_3)\in
X^3$ with $\Phi=G$ and $\tau=\tilde{\tau}$.
\end{itemize}
\end{lemma}

\begin{proof}
The implication $(ii)\Rightarrow(i)$ has already been proved in the
proof of Lemma \ref{l2.2*}.

Let us verify the implication $(i)\Rightarrow(ii)$. Suppose $(i)$
holds. It is sufficient to check that the equality
\begin{equation}\label{eq2.6*}
Im(x_1,x_2,x_3) = Im(y_1,y_2,y_3)
\end{equation}
implies
\begin{equation}\label{eq2.7*}
G(x_1,x_2,x_3) = G(y_1,y_2,y_3).
\end{equation}
Let \eqref{eq2.6*} hold. If $|Im(x_1,x_2,x_3)|=1$, then there is
$x\in X$ such that $x_i=x=y_i$ for every $i\in\{1,2,3\}$. In this
case \eqref{eq2.7*} transforms to the trivial equality
$G(x,x,x)=G(x,x,x)$. If $|Im(x_1,x_2,x_3)|=3$, then \eqref{eq2.7*}
follows from the invariance of $G$ with respect to the permutations
of arguments. For the case $|Im(x_1,x_2,x_3)|=2$ there are $x,y\in
X$ for which the triple $(x_1,x_2,x_3)$ coincides with one of the
triples $(x,x,y), (x,y,x), (y,x,x), (y,y,x), (y,x,y), (x,y,y)$. The
same holds for $(y_1,y_2,y_3)$ also. Now to prove \eqref{eq2.7*} we
can use \eqref{eq2.5*} and the invariance of $G$ with respect to the
permutations of arguments.
\end{proof}

\begin{remark}\label{r2.4*}
Since the mapping $X^3\ni(x_1,x_2,x_3)\mapsto
Im(x_1,x_2,x_3)\in\mathcal{F}^3(X)$ is surjective, the existence of
$\tilde{\tau}:\mathcal{F}^3(X)\to\mathbb{R}$ for which the equality
\begin{equation}\label{eq2.8*}
G(x_1,x_2,x_3) = \tilde{\tau}(Im(x_1,x_2,x_3))
\end{equation}
holds for every $(x_1,x_2,x_3)\in X^3$ implies the uniqueness of
$\tilde{\tau}$.
\end{remark}

\begin{lemma}\label{l2.5*}
Let $X\neq\varnothing$ and let $G:X^3\to\mathbb{R}$ be a symmetric
$G$-metric. Then there is an increasing mapping
$\tilde{\tau}:\mathcal{F}^3(X)\to\mathbb{R}$ such that: equality
\eqref{eq2.8*} holds for every $(x_1,x_2,x_3)\in X^3$; equivalence
\eqref{eq1.1} holds with $\tilde{\tau}=\tau$ for every
$A\in\mathcal{F}^3(X)$; inequality \eqref{eq1.2} holds with
$\tilde{\tau}=\tau$ for all $A\cup B, A\cup C, B\cup
C\in\mathcal{F}^3(X)$.
\end{lemma}

\begin{proof}
The existence of $\tilde{\tau}:\mathcal{F}^3(X)\to\mathbb{R}$ which
satisfies \eqref{eq2.8*} for $(x_1,x_2,x_3)\in X^3$ has already
proved in Lemma \ref{l2.3*}. The increase of $\tilde{\tau}$ and the
equivalence
$$
(\tilde{\tau}(A)=0)\Leftrightarrow(|A|=1)
$$
follow from conditions $(i)-(iii)$ of Definition \ref{def1.3} and
equality \eqref{eq2.8*}. We must prove only the inequality
\begin{equation}\label{eq2.9*}
\tilde{\tau}(A\cup B)\leqslant\tilde{\tau}(A\cup C) +
\tilde{\tau}(B\cup C)
\end{equation}
for $A\cup B, A\cup C, B\cup C\in\mathcal{F}^3(X)$.

Note that \eqref{eq2.9*} is trivial if $|A\cup B|=1$ because in
this case $\tilde{\tau}(A\cup B)=0$ holds. So we can suppose $|A\cup
B|=2$ or $|A\cup B|=3$. Since $\tilde{\tau}$ is increasing, it
suffices to prove \eqref{eq2.9*} for $C=\{a\}$ where $a$ is an
arbitrary point of $X$.

Let $|A\cup B|=2$. If, in addition, we have $|A|=2$, then
\begin{equation}\label{eq2.10*}
A\cup B = A \subseteq A\cup C.
\end{equation}
Hence using the nonnegativity of $G$ (see Remark \ref{r1.4*}) and
the increase of $\tilde{\tau}$ we obtain \eqref{eq2.9*}. If $|B|=2$,
then the proof is similar. Now let $A=\{x\}$, $B=\{y\}$ and $x\neq
y$. Then
\begin{equation}\label{eq2.11*}
\begin{gathered}
\tilde{\tau}(A\cup B)=G(x,y,y),\qquad\tilde{\tau}(A\cup
C)=G(x,a,a),\\
\tilde{\tau}(B\cup C) = G(y,a,a) = G(a,y,y).
\end{gathered}
\end{equation}
Putting $z=y$ in condition $(v)$ of Definition \ref{def1.3} we find
$$
G(x,y,y)\leqslant G(x,a,a) + G(a,y,y).
$$
This inequality and \eqref{eq2.11*} give \eqref{eq2.9*}.

Suppose $|A\cup B|=3$. If $\max(|A|,|B|)=3$, then we have
\eqref{eq2.10*} or $A\cup B \subseteq B\cup C$. Hence using the
increase of $\tau^*$ we get \eqref{eq2.9*}. If $|A|=2$ and $|B|=2$,
then there are some distinct $x,y,z\in X$ for which $A=\{x,y\}$ and
$B=\{y,z\}$. Consequently $\tilde{\tau}(A\cup B) = G(x,y,z)$,
$\tilde{\tau}(A\cup C) = G(x,y,a)$ and $\tilde{\tau}(B\cup C) =
G(y,z,a)$. Inequality \eqref{eq2.9*} can be rewritten as
\begin{equation}\label{eq2.12*}
G(x,y,z)\leqslant G(x,y,a) + G(y,z,a).
\end{equation}
Using condition $(iii)$ of Definition \ref{def1.3} and the
symmetry of $G$ we obtain the inequality
\begin{equation}\label{eq2.13*}
G(x,a,a)\leqslant G(x,y,a)
\end{equation}
for all $x,y,a\in X$. Now \eqref{eq2.13*} and condition $(v)$ of
Definition \ref{def1.3} imply \eqref{eq2.12*}. To complete the proof
it remains to consider the next alternative
$$
\text{ either }\quad |A|=1\text{ and } |B|=2\quad \text{ or }\quad
|B|=1\text{ and } |A|=2.
$$
By the symmetry of the occurrences of $A$ and $B$ in \eqref{eq2.9*}
it suffices to consider the first case. Putting $A=\{x\}$ and
$B=\{y,z\}$ and expressing $\tilde{\tau}$ via $G$, we get from
\eqref{eq2.9*} to the inequality $G(x,y,z)\leqslant G(x,a,a)+G(a,y,z)$.
Condition $(v)$ of Definition \ref{def1.3} claims the validity of the last inequality.
\end{proof}

In accordance with our plan it remains expand the function
$\tilde{\tau}:\mathcal{F}^3(X)\to\mathbb{R}$ to an increasing
extended by Balk metric $\tau:\mathcal{F}(X)\to\mathbb{R}$. It is
easy enough to do for all increasing $\tilde{\tau}:\mathcal{F}^k(X)\to\mathbb{R}$,
$\mathcal{F}^k(X) = \{A\in\mathcal{F}(X)\,:\,|A|\leqslant k\}$ with
an arbitrary integer $k\geqslant 2$.

For $A\in\mathcal{F}(X)$ and
$\tilde{\tau}:\mathcal{F}^k(X)\to\mathbb{R}$, $k\geqslant 1$ we set
\begin{equation}\label{eq2.16**}
\diam_{\tilde{\tau}}(A):=\max\{\tilde{\tau}(B)\,:\, B\subseteq A,
B\in\mathcal{F}^k(X)\},
\end{equation}
c.f. formula \eqref{n*}.

\begin{proposition}\label{pr2.6*}
Let $X\neq\varnothing$, let $k\geqslant 2$ be an integer number and
let $\tilde{\tau}:\mathcal{F}^k(X)\to\mathbb{R}$ be an increasing
mapping such that the equivalence
\begin{equation}\label{eq2.16*}
(\tilde{\tau}(A)=0)\Leftrightarrow (|A|=1)
\end{equation}
holds for each $A\in\mathcal{F}^k(X)$ and the inequality
\begin{equation}\label{eq2.17*}
\tilde{\tau}(A\cup B)\leqslant \tilde{\tau}(A\cup C) +
\tilde{\tau}(B\cup C)
\end{equation}
holds  as soon as $A\cup B, A\cup C, B\cup C\in\mathcal{F}^k(X)$.
Then the function
\begin{equation}\label{eq2.18*}
\tau:\mathcal{F}(X)\to\mathbb{R},\quad
\tau(A)=\diam_{\tilde{\tau}}(A),\quad A\in\mathcal{F}(X)
\end{equation}
is an increasing extended by Balk metric such that
\begin{equation}\label{eq2.19*}
\tau^k(x_1,\ldots,x_k) =
\tilde{\tau}(Im(x_1,\ldots,x_k))\quad\text{for}\quad
(x_1,\ldots,x_k)\in X^k.
\end{equation}
\end{proposition}

\begin{proof}
The increase of $\tau$ follows directly from equality
\eqref{eq2.16**}. This equality and the increase of $\tilde{\tau}$
give also equality \eqref{eq2.19*}. Using \eqref{eq2.16**} it is
easy to prove \eqref{eq2.16*} for every $A\in\mathcal{F}(X)$. It
remains to show that the inequality
\begin{equation}\label{eq2.20*}
\tau(A\cup B)\leqslant\tau(A\cup C) + \tau(B\cup C),
\end{equation}
holds for all $A,B,C\in\mathcal{F}(X)$.

Let $A,B$ and $C$ be arbitrary elements of $\mathcal{F}(X)$. Let us
choose an element $D\in\mathcal{F}^k(X)$ such that
\begin{equation}\label{eq2.21*}
D\subseteq A\cup B\text{ and }\tau(A\cup B)=\tilde{\tau}(D).
\end{equation}
If $D\subseteq A$ or $D\subseteq B$, then by increase of $\tau$ we
have $\tau(D)\leqslant\tau(A)\leqslant\tau(A\cup C)$ or,
respectively, $\tau(D)\leqslant\tau(B)\leqslant\tau(B\cup C)$. These
inequalities together with \eqref{eq2.21*} give \eqref{eq2.20*}.
Thus we can assume that
\begin{equation}\label{eq2.22*}
D\setminus A\neq \varnothing \neq D\setminus B.
\end{equation}
Set $A' := A\cap D$, $ B' := B\cap D$. Then we have
\begin{equation}\label{eq2.23*}
D=D\cap(A\cup B) = A'\cup B'.
\end{equation}
Condition \eqref{eq2.22*} and the inequality $|A|\leqslant k$ give
the inequalities
\begin{equation}\label{eq2.23**}
|A'|\leqslant k-1, \quad\text{and}\quad |B'|\leqslant k-1.
\end{equation}
Using \eqref{eq2.21*} and \eqref{eq2.23*} we write \eqref{eq2.20*}
in the form
$$
\tau(A'\cup B')\leqslant\tau(A\cup C)+\tau(B\cup C).
$$
Since $\tau$ is increasing, $A'\subseteq A$ and $B'\subseteq B$, it
suffices to check the inequality
$$
\tau(A'\cup B')\leqslant\tau(A'\cup C)+\tau(B'\cup C).
$$
Let $C'=\{c\}$ where $c\in C$. Since $\tau$ is increasing, we have
$$
\tau(A'\cup C) + \tau(B'\cup C) \geqslant \tau(A'\cup
C')+\tau(B'\cup C').
$$
Therefore it is sufficient to show that
\begin{equation}\label{eq2.25*}
\tau(A'\cup B')\leqslant\tau(A'\cup C')+\tau(B'\cup C').
\end{equation}
To prove \eqref{eq2.25*} note that \eqref{eq2.23**} implies that
$$
|A'\cup C'|\leqslant |A'| + |C'|\leqslant (k-1)+1=k,
$$
and, similarly that $|B'\cup C'|\leqslant k$. Thus $A'\cup C', B'\cup
C'\in\mathcal{F}^k(X)$. In addition we have $A'\cup B' =
D\in\mathcal{F}^k(X)$. Now using \eqref{eq2.19*} we can rewrite
\eqref{eq2.25*} in the form $\tilde{\tau}(A'\cup
B')\leqslant\tilde{\tau}(A'\cup C')+\tilde{\tau}(B'\cup C')$ that
holds by \eqref{eq2.17*}.

Thus the function $\tau$ defined on $\mathcal{F}(X)$ by formula
\eqref{eq2.18*} has all properties of extended by Balk metric.
\end{proof}

\begin{remark}\label{r2.7*}
Proposition \ref{pr2.6*} is false for $k=1$. In this case we have
$\diam_{\tilde{\tau}}(A)=0$ for every $A\in\mathcal{F}(X)$.
\end{remark}

\begin{theorem}\label{t2.9*}
Let $X\neq\varnothing$. The following statements are equivalent for
every function $G:X^3\to\mathbb{R}$.
\begin{itemize}
\item[$(i)$] $G$ is a symmetric $G$-metric in the sense of Definition \ref{def1.3}.

\item[$(ii)$] There is an increasing extended by Balk metric
$\tau:\mathcal{F}(X)\to\mathbb{R}$ such that $\tau^3=G$.
\end{itemize}
\end{theorem}

\begin{proof}
The implication $(ii)\Rightarrow(i)$ was obtained in Lemma
\ref{l2.2*}. Let us prove the implication $(i)\Rightarrow(ii)$. Suppose
$(i)$ holds. By Lemma \ref{l2.3*} there is
$\tilde{\tau}:\mathcal{F}^3(X)\to\mathbb{R}$ such that
$G(x_1,x_2,x_3)=\tilde{\tau}(Im(x_1,x_2,x_3))$ for every
$(x_1,x_2,x_3)\in X^3$. Using Lemma \ref{l2.5*} we get the
equivalence $(\tilde{\tau}(A)=0)\Leftrightarrow (|A|=1)$ for every
$A\in\mathcal{F}^3(X)$ and the inequality $\tilde{\tau}(A\cup
B)\leqslant\tilde{\tau}(A\cup C)+\tilde{\tau}(B\cup C)$ for $A\cup
B, A\cup C, B\cup C\in\mathcal{F}^3(X)$. By Proposition \ref{pr2.6*}
there is an increasing extended by Balk metric
$\tau:\mathcal{F}(X)\to\mathbb{R}$ such
that$\tau|_{\mathcal{F}^3(X)}=\tilde{\tau}$. The implication
$(i)\Rightarrow(ii)$ is proved.
\end{proof}

The next theorem is an analog of Theorem \ref{t2.13}.

\begin{theorem}\label{t2.15}
Let $\tau:\mathcal{F}(X)\to\mathbb{R}$ be an extended by Balk metric
on $X\neq\varnothing$. The following statements are equivalent.
\begin{itemize}
\item[$(i)$] There is a function $G:X^3\to\mathbb{R}$ such that the equality
\begin{equation}\label{eq2.26*}
\tau(A)=\max\{G(x,y,z)\,:\,x,y,z\in A\}
\end{equation}
holds for every $A\in\mathcal{F}(X)$.
\item[$(ii)$] There is a symmetric $G$-metric $G:X^3\to\mathbb{R}$ such that
equality \eqref{eq2.26*} holds for every $A\in\mathcal{F}(X)$.
\item[$(iii)$] For every $A\in\mathcal{F}(X)$ the equality \eqref{eq2.26*}
holds with $G=\tau^3$.
\item[$(iv)$]  $\tau$ is 3-increasing and 3-weakly decreasing.
\item[$(v)$]  $\tau$ is increasing and 3-weakly decreasing.
\end{itemize}
\end{theorem}

\begin{proof}
If $(iii)$ holds, then $\tau$ is increasing.
Then, by Lemma \ref{l2.2*}, $\tau^3$ is symmetric $G$-metric.
Therefore the implication $(iii)\Rightarrow(ii)$ is true. The
implication $(ii)\Rightarrow(i)$ is obvious. The implication
$(i)\Rightarrow(v)$ follows directly from definitions. To complete
the proof it remains to note that Theorem \ref{t2.11} gives
$(v)\Leftrightarrow(iv)$ and $(iv)\Leftrightarrow(iii)$.
\end{proof}

\section{Extended metrics on pretangent spaces}

Let $(X,d)$ be a metric space with a metric $d$. The infinitesimal
geometry of the space $X$ can be investigated by constructing of
metric spaces that, in some sense, are tangent to $X$. If $X$ is
equipped with an additional structure, then the question arises of
the lifting this structure on the tangent spaces. More specifically,
let $\tau:\mathcal{F}(X)\to\mathbb{R}$ be an extended by Balk metric
and let $(\Omega,\rho)$ be a tangent space to a metric space
$(X,d)$. Suppose $\tau$ is compatible with $d$.

{\it How to build an extended by Balk metric which is
compatible with the metric $\rho$?}

The answer to this question depends on the construction of the tangent
space $(\Omega,\rho)$. Today there are several approaches to construct
tangent spaces to metric spaces. Probably, the most famous of these are
the Gromov-Hausdorff convergence and the ultra-convergence. The sequential
approach to the construction of "pretangent" and "tangent" spaces was
proposed in \cite{DM1} and developed in \cite{DAK1,DAK2,BD2,Bil,Dor, Dov}.

To construct Balk's extended metrics on pretangent spaces we will
use ultrafilters on $\mathbb{N}$. Recall the necessary definitions.

Let $X$ be a nonempty set and let $\mathfrak{B}(X)$ be the set of all
its subsets. A set $\mathfrak{U}\subseteq\mathfrak{B}(X)$ is
called a {\it filter} on $X$ if
$\varnothing\notin\mathfrak{U}\neq\varnothing$ and the implication
$$
(A\in\mathfrak{U}\,\&\,B\in\mathfrak{U})\Rightarrow (A\cap
B\in\mathfrak{U}))\quad\text{and}\quad(A\supseteq
B\,\&\,B\in\mathfrak{U})\Rightarrow (A\in\mathfrak{U}))
$$
hold for all $A,B\subseteq X$.

A filter $\mathfrak{U}$ on $X$ is called an {\it ultrafilter} if the
implication
$(\mathcal{P}\supseteq\mathfrak{U})\Rightarrow(\mathfrak{U}=\mathcal{P})$
holds for every filter $\mathcal{P}$ on $X$. An ultrafilter
$\mathfrak{U}$ on $X$ is called {\it trivial} if there is a point
$x_0\in X$ such that $(A\in\mathfrak{U})\Leftrightarrow(x_0\in A)$
for $A\in\mathfrak{B}(X)$. Otherwise $\mathfrak{U}$ is a nontrivial
ultrafilter. Let  $\mathfrak{U}$ be a filter on $X$. A mapping
$\Phi:X\to\mathbb{R}$ converges to a point $t\in\mathbb{R}$ by the
filter $\mathfrak{U}$, symbolically $\mathfrak{U}-\lim\Phi(x)=t$, if
\begin{equation}\label{eq3.3*}
\{x\in X\,:\, |\Phi(x)-t|<\varepsilon\}\in\mathfrak{U}
\end{equation}
for every $\varepsilon>0$.

\begin{example}\label{ex3.4}
If $\mathcal{M}$ is the Frechet filter on $\mathbb{N}$, (is the
family of subsets of $\mathbb{N}$ with finite complements) and
$(x_n)_{n\in\mathbb{N}}$ is a sequence of real numbers, then the
limit $\underset{n\to\infty}{\lim}x_n$ exists if and only if there
is $\mathcal{M}-\lim x_n$. In this case
$\underset{n\to\infty}{\lim}x_n = \mathcal{M}-\lim x_n$.
\end{example}

We shall use the following properties of the nontrivial ultrafilters
$\mathfrak{U}$ on $\mathbb{N}$.
\begin{itemize}
\item[$(i_1)$] Every bounded sequence
$(x_n)_{n\in\mathbb{N}}$, $x_n\in\mathbb{R}$, has
$\mathfrak{U}-\lim x_n$;

\item[$(i_2)$] If $(x_n)_{n\in\mathbb{N}}$ converges in the usual
sense, then $\underset{n\to\infty}{\lim}x_n = \mathfrak{U}-\lim
x_n$;

\item[$(i_3)$] The relations
$\mathfrak{U}-\lim cx_n = c (\mathfrak{U}-\lim x_n)$ and
$$
\mathfrak{U}-\lim (x_n+y_n) =
(\mathfrak{U}-\lim x_n) + (\mathfrak{U}-\lim y_n)
$$
hold for every $c\in\mathbb{R}$ and $(x_n)_{n\in\mathbb{N}}$,
$(y_n)_{n\in\mathbb{N}}$ which have $\mathfrak{U}$-limits;

\item[$(i_4)$] If $(x_n)_{n\in\mathbb{N}}$ is
$\mathfrak{U}$-convergent and $\underset{n\to\infty}{\lim}(x_n -y_n)
= 0$, then $(y_n)_{n\in\mathbb{N}}$ is $\mathfrak{U}$-convergent and
$\mathfrak{U}-\lim x_n = \mathfrak{U}-\lim y_n$.
\end{itemize}

The above is a trivial modification of Problem 19 from Chapter 17
\cite{KT}.

\begin{lemma}\label{l3.4*}
Let $\mathfrak{U}$ be a nontrivial ultrafilter on $\mathbb{N}$. Then
for every bounded sequence $(x_m)_{m\in\mathbb{N}}$,
$x_m\in\mathbb{R}$ its $\mathfrak{U}$-limit coincides with a limit
point of this sequence. Conversely, if $t$ is a limit
point of $(x_m)_{m\in\mathbb{N}}$, then there is a nontrivial
ultrafilter $\mathfrak{U}$ on $\mathbb{N}$ such that
$\mathfrak{U}-\lim x_m = t$.
\end{lemma}

\begin{proof}

The first statement of the lemma follows from the definition of the
limit points and formula \eqref{eq3.3*} if put $X=\mathbb{N}$ and
$\Phi(n)=x_n$ in this formula and take into account that all
elements of nontrivial ultrafilter on $\mathbb{N}$ are infinite
subsets of $\mathbb{N}$.

To prove the second statement, note that for every limit point $a$
of the sequence $(x_m)_{m\in\mathbb{N}}$ there is an infinite
$A\subseteq\mathbb{N}$ such that $\underset{\underset{m\in
A}{m\to\infty}}{\lim}x_m=a$. Choose an ultrafilter
$\mathbf{\mathfrak{U}}$ on $\mathbb{N}$ for which
$A\in\mathbf{\mathfrak{U}}$. Now using property $(i_1)$ we obtain
$a=\mathfrak{U}-\lim x_m$.
\end{proof}

Let $(X,d)$ be a metric space, $X\neq\varnothing$, and let
$\tau:\mathcal{F}(X)\to\mathbb{R}$ be a compatible with $d$ extended
by Balk metric. Let $\{\alpha_1,\ldots,\alpha_n\}$ be a finite
nonempty subset of pretangent space $\Omega^X_{p,\tilde{r}}$ and
$\tilde{X}_{p,\tilde{r}}$ be a maximal self-stable subset of
$\tilde{X}_p$ which corresponds $\Omega^X_{p,\tilde{r}}$. Denote by
$\pi$ the projection $\tilde{X}_{p,\tilde{r}}$ on
$\Omega^X_{p,\tilde{r}}$, i.e. if
$\tilde{x}\in\tilde{X}_{p,\tilde{r}}$ then
$\pi(\tilde{x})=\{\tilde{y}\in\tilde{X}_{p,\tilde{r}}\,:\,
\tilde{d}_{\tilde{r}}(\tilde{x},\tilde{y})=0\}$. (See formula
\eqref{eq3.1}). Choose $\tilde{x}^i=(x^i_m)_{m\in\mathbb{N}}$,
$i=1,\ldots,n$ from $\tilde{X}_{p,\tilde{r}}$ such that
$\pi(\tilde{x}^i)=\alpha_i,\quad i=1,\ldots,n$. Put

\begin{equation}\label{eq3.3}
\mathcal{X}_{\tau}(\{\alpha_1,\ldots,\alpha_n\}):=\mathfrak{U}-\lim\frac{\tau(Im(x^1_m,\ldots,x_m^n))}{r_m},
\end{equation}
where $Im(x^1_m,\ldots,x_m^n)$ is the set whose elements are the $m$-th
terms of the sequences $(x^i_m)_{m\in\mathbb{N}}$, $i=1,\ldots,n$,
$r_m$ is $m$-th term of normalizing sequence $\tilde{r}$ and
$\mathfrak{U}$ is a nontrivial ultrafilter on $\mathbb{N}$.

\begin{theorem}\label{t3.5}
Let $(X,d,p)$ be a metric space with a marked point $p$ and
$\tau:\mathcal{F}(X)\to\mathbb{R}$ be an extended by Balk metric. If
$\tau$ is compatible with $d$, then for every pretangent space
$(\Omega^X_{p,\tilde{r}},\rho)$ and every nontrivial ultrafilter
$\mathfrak{U}$ on $\mathbb{N}$ the mapping
$$
\mathcal{F}(\Omega^X_{p,\tilde{r}})\ni A\mapsto
\mathcal{X}_{\tau}(A)
$$
is correctly defined  extended by Balk metric which is compatible
with the metric $\rho$.
\end{theorem}

To prove this theorem we need the next lemma.

\begin{lemma}\label{l3.6}
Let $(X,d)$ be a nonempty metric space and
$\tau:\mathcal{F}(X)\to\mathbb{R}$ be an extended by Balk metric. If
$\tau$ is compatible with $d$, then the inequalities
\begin{equation}\label{eq3.4}
\tau(\{x_1,x_2,\ldots,x_n\})\leqslant
d(x_1,x_2)+\ldots+d(x_{n-1},x_n),
\end{equation}
and
\begin{equation}\label{eq3.5}
|\tau(\{x_1,\ldots,x_n\})-\tau(\{x'_1,\ldots,x'_n\})|\leqslant
\underset{i=1}{\overset{n}{\sum}}d(x_i,x'_i),
\end{equation}
hold for every integer $n\geqslant 1$. Here $\{x_1,\ldots,x_n\}$ and
$\{x'_1,\ldots,x'_n\}$ are arbitrary $n$-elements subsets of the set
$X$.
\end{lemma}

\begin{proof}
Without loss of generality we can suppose that $n\geqslant 2$. Let
$\{x_1,\ldots,x_n\}\in\mathcal{F}(X)$. Using \eqref{eq2.1} with
$B=\{x_n\}$, $C=\{x_{n-1}\}$ and $A=\{x_1,\ldots,x_{n-1}\}$ we find
$$
\tau(\{x_1,\ldots,x_n\})\leqslant\tau(\{x_1,\ldots,x_{n-1}\}) +
\tau(\{x_{n-1},x_n\})\leqslant \tau(\{x_1,\ldots,x_{n-1}\}) +
d(x_{n-1},x_n).
$$
Repeating this procedure we obtain inequality \eqref{eq3.4}.

Let us check \eqref{eq3.5}. To this end note that
$$
\begin{gathered}
|\tau(\{x_1,\ldots,x_n\})-\tau(\{x'_1,\ldots,x'_n\})|
\\
\leqslant |\tau(Im(x_1,x_2\ldots,x_n))-\tau(Im(x'_1,x_2\ldots,x_n))|
\\
+ |\tau(Im(x'_1,x_2\ldots,x_n))-\tau(Im(x'_1,x'_2,x_3\ldots,x_n))|
\ldots \\
+
|\tau(Im(x'_1,x'_2\ldots,x'_{n-1},x_n))-\tau(Im(x'_1,\ldots,x'_{n-1},x'_n))|.
\end{gathered}
$$
The sets, which are the arguments of the function $\tau$ under the
signs of the absolute value on the right-hand side of the last
inequality, differ from each other by no more than one element.
Therefore, it suffices to verify the inequality
\begin{equation}\label{eq3.6}
|\tau(Im(x'_1,\ldots,x'_{n-1},x_n))-\tau(Im(x'_1,\ldots,x'_{n-1},x'_n))|
\leqslant d(x_n,x'_n).
\end{equation}
\noindent Without loss of generality we can suppose that
\begin{equation}\label{eq3.7}
\tau(Im(x'_1,\ldots,x'_{n-1},x_n))\geqslant\tau(Im(x'_1,\ldots,x'_{n-1},x'_n)).
\end{equation}
\noindent Using inequality \eqref{eq1.2} with
$A=\{x'_1,\ldots,x'_{n-1}\}$, $B=\{x_n\}$, $C=\{x'_n\}$ we get
$$
A\cup B=Im(x'_1,\ldots,x'_{n-1},x_n),\quad A\cup
C=Im(x'_1,\ldots,x'_{n-1},x'_n), \quad B\cup C=Im(x_n,x'_n)
$$
and
$$
\tau(Im(x'_1,\ldots,x'_{n-1},x_n))-\tau(Im(x'_1,\ldots,x'_{n-1},x'_n))
\leqslant\tau(\{x_n,x'_n\})=d(x_n,x'_n).
$$
The last inequality together with \eqref{eq3.7} gives \eqref{eq3.6}.
\end{proof}

\begin{lemma}\label{l3.7}
Let $(X,d)$ be a nonempty metric space, let
$\tau:\mathcal{F}(X)\to\mathbb{R}$ be an extended by Balk metric and
let $K\in\mathcal{F}(X)$. If $\tau$ is compatible with $d$, then the
inequality
\begin{equation}\label{eq3.9*}
\tau(K)\geqslant\frac{1}{2}d(x,y)
\end{equation}
holds for all $x,y\in K$.
\end{lemma}

\begin{proof}
Let $A=\{x\}$, $B=\{y\}$ and $C=K$. Then by inequality \eqref{eq1.2}
we have
$$
d(x,y)=\tau(A\cup B)\leqslant\tau(A\cup K) + \tau(B\cup K)=2\tau(K).
$$
\end{proof}

\begin{proof}[\it The proof of the Theorem \ref{t3.5}.] Let us check
the  existence of the finite $\mathfrak{U}$-limit on the right-hand
side of \eqref{eq3.3}. According to property $(i_1)$ of the
ultrafilters it suffices to prove the inequality
\begin{equation}\label{eq3.9}
\underset{m\to\infty}{\lim\sup}\,\frac{\tau(Im(x^1_m,\ldots,x_m^n))}{r_m}
< \infty.
\end{equation}

Using \eqref{eq3.1}, \eqref{eq3.2} and \eqref{eq3.4} we find
$$
\begin{gathered}
\underset{m\to\infty}{\lim\sup}\,\frac{\tau(Im(x^1_m,\ldots,x_m^n))}{r_m}\leqslant
\underset{m\to\infty}{\lim}\,\frac{d(x^1_m,x^2_m)}{r_m}
\\ + \underset{m\to\infty}{\lim}\,\frac{d(x^2_m,x^3_m)}{r_m} + \ldots +
\underset{m\to\infty}{\lim}\,\frac{d(x^{n-1}_m,x^n_m)}{r_m} =
\underset{i=1}{\overset{n-1}{\sum}}\rho(\alpha_i, \alpha_{i+1}).
\end{gathered}
$$
Inequality \eqref{eq3.9} follows.

Let us make sure that the value of
$\mathcal{X}_{\tau}(\{\alpha_1,\ldots,\alpha_n\})$ given in formula
\eqref{eq3.3} does not depend of the choice of $\tilde{x}^i$,
$i=1,\ldots,n$. Let $\tilde{y}^i=(y^i_m)$ be some elements of the
set $\tilde{X}_{p,\tilde{r}}$ such that
$\pi(\tilde{y}^i)=\pi(\tilde{x}^i)=\alpha_i$, $i=1,\ldots,n$.

By \eqref{eq3.5} we have
$$
\begin{gathered}
\underset{m\to\infty}{\lim\sup}\,\frac{\tau(Im(x^1_m,\ldots,x_m^n))-\tau(Im(y^1_m,\ldots,y_m^n))}{r_m}
\\\leqslant
\underset{i=1}{\overset{n}{\sum}}
\left(\underset{m\to\infty}{\lim}\,\frac{d(x^i_m,y^i_m)}{r_m}\right)=
\underset{i=1}{\overset{n}{\sum}}\rho(\alpha_i,\alpha_i)=0.
\end{gathered}
$$The wanted independence follows from $(i_4)$.

Let us verify that
$\mathcal{X}_{\tau}:\mathcal{F}(\Omega^X_{p,\tilde{r}})\to\mathbb{R}$
has the characteristic properties of extended metric i.e.,

\begin{equation}\label{eq3.10}
(\mathcal{X}_{\tau}(A)=0) \Leftrightarrow (|A|=1)
\end{equation}
and
\begin{equation}\label{eq3.11}
\mathcal{X}_{\tau}(A\cup B)\leqslant \mathcal{X}_{\tau}(A\cup C) +
\mathcal{X}_{\tau}(B\cup C)
\end{equation}
hold for all $A,B,C\in\mathcal{F}(\Omega^X_{p,\tilde{r}})$.

Let $|A|=1$. Then we have $A=\{\alpha\}$ for some
$\alpha\in\Omega^X_{p,\tilde{r}}$. If
$\tilde{x}=(x_m)_{m\in\mathbb{N}}\in\tilde{X}_{p,\tilde{r}}$ and
$\pi(\tilde{x})=\alpha$, then \eqref{eq3.3} and property $(i_2)$ of
the ultrafilters imply
$$
\mathcal{X}_{\tau}(A)=\mathfrak{U}-\lim\frac{\tau(\{x_m\})}{r_m}=\mathfrak{U}-\lim
0 = 0.
$$
Suppose now that $A$ has at least two distinct points
$\alpha_1=\pi((x^1_m)_{m\in\mathbb{N}})$ and
$\alpha_2=\pi((x^2_m)_{m\in\mathbb{N}})$ where
$(x^1_m)_{m\in\mathbb{N}},(x^2_m)_{m\in\mathbb{N}}\in\tilde{X}_{p,\tilde{r}}$.
Then inequality \eqref{eq3.9*} implies
$$
\frac{\tau(Im(x^1_m,\ldots,x^n_m))}{r_m}\geqslant\frac{1}{2}\,\frac{d(x^1_m,x^2_m)}{r_m}
$$
for every $m\in\mathbb{N}$. By the definition we have
$$
\underset{m\to\infty}{\lim}\,\frac{1}{2}\,\frac{d(x^1_m,x^2_m)}{r_m}
= \frac{1}{2}\rho(\alpha_1,\alpha_2)>0.
$$

\noindent Hence all limit points of the sequence with the common
term
$$
\frac{\tau(Im(x^1_m,\ldots,x^n_m))}{r_m}
$$
are positive. Therefore, by Lemma \ref{l3.4*}, we obtain the strict inequality
$\mathcal{X}_{\tau}(A)>0$. Equivalence \eqref{eq3.10} is proved.

Similarly, considering the limit points of the sequence that defines
the value
$$
(\mathcal{X}_{\tau}(A\cup B) - \mathcal{X}_{\tau}(A\cup C) -
\mathcal{X}_{\tau}(B\cup C))
$$
and using \eqref{eq1.2} and $(i_3)$ we obtain \eqref{eq3.11}. Thus
$\mathcal{X}_{\tau}$ is an extended by Balk metric on
$\Omega^X_{p,\tilde{r}}$.

To complete the proof it remains to check that $\mathcal{X}_{\tau}$
is compatible with $\rho$. Let
$\alpha_1,\alpha_2\in\Omega^X_{p,\tilde{r}}$ and
$\alpha_1=\pi((x^1_m)_{m\in\mathbb{N}}),
\alpha_2=\pi((x^2_m)_{m\in\mathbb{N}})$ where
$(x^1_m)_{m\in\mathbb{N}},(x^2_m)_{m\in\mathbb{N}}\in\tilde{X}_{p,\tilde{r}}$.
Then from \eqref{eq3.1}, \eqref{eq3.2}, \eqref{eq3.3} and the fact
that $\tau$ is compatible with $d$ we find
$$
\begin{gathered}
\mathcal{X}_{\tau}(\{\alpha_1,\alpha_2\})=
\mathfrak{U}-\lim\frac{\tau(Im(x^1_m,x^2_m))}{r_m}
\\
=\mathfrak{U}-\lim\frac{d(x^1_m,x^2_m)}{r_m}=
\underset{m\to\infty}{\lim}\frac{d(x^1_m,x^2_m)}{r_m}=\rho(\alpha_1,\alpha_2),
\end{gathered}
$$
which is what had to be proved.
\end{proof}

It is rather easy to show if an extended by Balk metric
$\tau:\mathcal{F}(X)\to\mathbb{R}$ is generated by a metric
$d:X^2\to\mathbb{R}$, then for all pretangent spaces
$(\Omega^X_{p,\tilde{r}},\rho)$ and nontrivial ultrafilters
$\mathfrak{U}$ the extended metrics $\mathcal{X}_{\tau}$ are
generated by $\rho$. On the other hand if the space $(X,d)$ is
discrete, then every pretangent space $\Omega^X_{p,\tilde{r}}$ is
single-point. Consequently $\mathcal{X}_{\tau}$ is generated by the
metric $\rho$ as the extended by Balk metric on the single-point
space, irrespective of whether $\tau$ is generated by the metric $d$
or not.

To describe the class of extended metrics $\tau:X\times
X\to\mathbb{R}$ for which $\mathcal{X}_{\tau}$ is generated by $\rho$,
we need some "infinitesimal" variant of Definition
\ref{def1.2}.

Let $(X,d)$ be a metric space, $p\in X$ and
$\tau:\mathcal{F}(X)\to\mathbb{R}$ be an extended metric for which
$\tau^{2}=d$, i.e. $\tau$ is compatible with $d$.

\begin{definition}\label{def3.8}
The extended metric $\tau$ is generated by the metric $d$ at the
point $p$ if for every $n\in\mathbb{N}$ and every finite set of
sequences $(x^1_m)_{m\in\mathbb{N}},\ldots,
(x^n_m)_{m\in\mathbb{N}}$ which converge to $p$ with $m\to\infty$,
the relation
\begin{equation}\label{eq3.13}
|\tau(Im(x^1_m,\ldots,x^n_m))-\diam_d(Im(x^1_m,\ldots,x^n_m))| =
o(\max\{d(x^1_m,p),\ldots,d(x^n_m,p)\})
\end{equation}
holds, where $Im(x^1_m,\ldots,x^n_m)$ is defined by \eqref{eq1.6}.
\end{definition}

\begin{remark}\label{r3.9}
Relation \eqref{eq3.13} means that
\begin{equation}\label{eq3.14}
\underset{m\to\infty}{\lim}\frac{|\tau(Im(x^1_m,\ldots,x^n_m))-\diam_d(Im(x^1_m,\ldots,x^n_m))|}
{\max\{d(x^1_m,p),\ldots,d(x^n_m,p)\}}=0,
\end{equation}
with
$\underset{m\to\infty}{\lim}\frac{|\tau(Im(x^1_m,\ldots,x^n_m))-\diam_d(Im(x^1_m,\ldots,x^n_m))|}
{\max\{d(x^1_m,p),\ldots,d(x^n_m,p)\}}:=0$ for
$x^1_m=\ldots=x^n_m=p$.
\end{remark}

\begin{theorem}\label{t3.10}
Let $X\neq\varnothing$, $p\in X$ and let
$\tau:\mathcal{F}(X)\to\mathbb{R}$ be an extended by Balk metric.
Then the following statements are equivalent.
\begin{itemize}
\item[$(i)$] For every nontrivial ultrafilter $\mathfrak{U}$
on $\mathbb{N}$ and every space $(\Omega^X_{p,\tilde{r}},\rho)$
which is pretangent to the metric space $(X,\tau^2)$ at the point
$p$, the extended metric
$\mathcal{X}_{\tau}:\mathcal{F}(\Omega^X_{p,\tilde{r}})\to\mathbb{R}$
is generated by $\rho$.

\item[$(ii)$] The extended metric $\tau$ is generated by the metric $\tau^2$ at the point $p$.
\end{itemize}
\end{theorem}

\begin{proof}
For convenience write $d:=\tau^2$.

First consider the case when $p$ is an isolated point of the space
$(X,d)$. In this case the equality $x_m=p$ holds for every sequence
$(x_m)_{m\in\mathbb{N}}\in\tilde{X}_p$ if $m\in\mathbb{N}$ is
sufficiently large. Using Remark \ref{r3.9} we see that $\tau$ is
generated by $d$ at the point $p$. For isolated $p$ any
$\Omega^X_{p,\tilde{r}}$ is a single-point space. Property $(ii)$
holds. To prove $(i)$ observe, that for the single-point space there
is the unique extended metric which is generated by the unique
metric on such space.

Let us turn now to less trivial case when $p$ is a limit point of
$X$. Suppose $(ii)$ holds. Consider an arbitrary pretangent space
$(\Omega^X_{p,\tilde{r}},\rho)$, a nontrivial ultrafilter
$\mathfrak{U}$ on $\mathbb{N}$ and the extended metric
$\mathcal{X}_{\tau}:\mathcal{F}(\Omega^X_{p,\tilde{r}})\to\mathbb{R}$
which is defined by $\mathfrak{U}$ according to \eqref{eq3.3}.

We must prove the equality
\begin{equation}\label{eq3.15}
\mathcal{X}_{\tau}(A)=\diam_{\rho}A
\end{equation}
for arbitrary
$A=\{\alpha_1,\ldots,\alpha_n\}\in\mathcal{F}(\Omega^X_{p,\tilde{r}})$.

By Theorem \ref{t3.5}, $\mathcal{X}_{\tau}$ is compatible with
$\rho$. Therefore \eqref{eq3.15} holds for $n\leqslant 2$. So we can
assume $n\geqslant 3$. Let $\tilde{X}_{p,\tilde{r}}$ be a maximal
self-stable family in $\tilde{X}_p$ corresponding to
$\Omega^X_{p,\tilde{r}}$ and let
$\pi:\tilde{X}_{p,\tilde{r}}\to\Omega^X_{p,\tilde{r}}$ be the
projection that maps the sequences $(x_m)_{m\in\mathbb{N}}\in
\tilde{X}_{p,\tilde{r}}$ to their equivalence classes (see
\eqref{eq3.2}). Relation \eqref{eq3.15} can be rewritten as
\begin{equation}\label{eq3.16}
\mathfrak{U}-\lim\frac{\tau(Im(x^1_m,\ldots,x^n_m))}{r_m}=\diam_{\rho}A,
\end{equation}
where $r_m$ and $x^1_m,\ldots,x^n_m$ are the $m$-th elements of the
normalizing sequence $\tilde{r}=(r_m)_{m\in\mathbb{N}}$ and,
respectively, of the sequences
$(x^1_m)_{m\in\mathbb{N}},\ldots,(x^n_m)_{m\in\mathbb{N}}\in\tilde{X}_{p,\tilde{r}}$
for which $\pi((x^i_m))=\alpha_i$, $i=1,\ldots,n$. The following
limit relations directly follow from the definition of the metric
$\rho$ on $\Omega^X_{p,\tilde{r}}$,

\begin{equation}\label{eq3.16*}
\begin{gathered}
\diam_{\rho}A=\underset{m\to\infty}{\lim}\frac{\diam_d(Im(x^1_m,\ldots,x^n_m))}{r_m},\\
\max(\{\rho(\alpha,\alpha_1),\ldots,\rho(\alpha,\alpha_n)\})=
\underset{m\to\infty}{\lim}\frac{\max(\{d(p,x^1_m),\ldots,d(p,x^n_m)\})}{r_m},
\end{gathered}
\end{equation}
where $\alpha=\pi(\tilde{p})$, $\tilde{p}=(p,p,p\ldots)$. Using
properties $(i_2)$, $(i_3)$ of the ultrafilters, equalities
\eqref{eq3.2} and \eqref{eq3.3} and the first equality from
\eqref{eq3.16*}, we can rewrite \eqref{eq3.16} in the form
\begin{equation}\label{eq3.17}
\mathfrak{U}-{\lim}\frac{\tau(Im(x^1_m,\ldots,x^n_m))-\diam_d(Im(x^1_m,\ldots,x^n_m))}
{r_m}=0.
\end{equation}
Since $n\geqslant 3$, then the strict inequality
$$
\max\{d(x^1_m,p),\ldots,d(x^n_m,p)\}>0
$$
holds for sufficient large $m$. Using this inequality, \eqref{eq3.14} and the second equality from
\eqref{eq3.16*} we find
$$
\begin{gathered}
\underset{m\to\infty}{\lim}\frac{\tau(Im(x^1_m,\ldots,x^n_m))-\diam_d(Im(x^1_m,\ldots,x^n_m))}
{r_m}\\
=\underset{m\to\infty}{\lim}\left(\frac{\tau(Im(x^1_m,\ldots,x^n_m))-\diam_d(Im(x^1_m,\ldots,x^n_m))}
{\max\{d(x^1_m,p),\ldots,d(x^n_m,p)\}}\,\frac{\max\{d(x^1_m,p),\ldots,d(x^n_m,p)\}}{r_m}\right)
\\
=0\cdot\max\{\rho(\alpha,\alpha_1),\ldots,\rho(\alpha,\alpha_n)\}=0.
\end{gathered}
$$
Now property $(i_2)$ of the ultrafilters implies \eqref{eq3.17}.
The implication $(ii)\Rightarrow(i)$ is proved.

To complete the proof it remains to establish the
converse implication $(i)\Rightarrow(ii)$. Suppose that $(i)$ is
true but $(ii)$ is false. Then there are an integer number
$n\geqslant 3$ and sequences
$(x^i_m)_{m\in\mathbb{N}}\in\tilde{X}_p$, $i=1,\ldots,n$ such that a
limit point $b$ of the sequence $(y_m)_{m\in\mathbb{N}}$,
$$
y_m=\frac{\tau(Im(x^1_m,\ldots,x^n_m))-\diam_d(Im(x^1_m,\ldots,x^n_m))}
{\max\{d(x^1_m,p),\ldots,d(x^n_m,p)\}},
$$
is nonzero, $b\neq 0$. The sequence $(y_m)_{m\in\mathbb{N}}$ is
bounded. Indeed, if
$\diam_d(Im(x^1_m,\ldots,x^n_m))=d(x_m^{i_1},x_m^{i_2})$,
$1\leqslant i_1, i_2\leqslant n$, then
\begin{equation}\label{eq3.18}
0\leqslant\frac{\diam_d(Im(x^1_m,\ldots,x^n_m))}
{\max\{d(x^1_m,p),\ldots,d(x^n_m,p)\}}\leqslant
\frac{d(x_m^{i_1},p)+d(x_m^{i_2},p)}
{\max\{d(x^1_m,p),\ldots,d(x^n_m,p)\}}\leqslant 2.
\end{equation}

Similarly, using \eqref{eq3.4} we find
$$
\tau(Im(x^1_m,\ldots,x^n_m))\leqslant\underset{i=1}{\overset{n-1}{\sum}}\tau(Im(x^i_m,x^{i+1}_m))
=\underset{i=1}{\overset{n-1}{\sum}}d(x^i_m,x^{i+1}_m),
$$
that gives
\begin{equation}\label{eq3.19}
\begin{gathered}
0\leqslant\frac{\tau(Im(x^1_m,\ldots,x^n_m))}{\max\{d(x^1_m,p),\ldots,d(x^n_m,p)\}}\leqslant
\underset{i=1}{\overset{n-1}{\sum}}\frac{d(x^i_m,x^{i+1}_m)}{\max\{d(x^1_m,p),\ldots,d(x^n_m,p)\}}
\\\leqslant
\underset{i=1}{\overset{n-1}{\sum}}\frac{d(x^i_m,p)+d(x^{i+1}_m,p)}{\max\{d(x^1_m,p),\ldots,d(x^n_m,p)\}}
\leqslant 2(n-1).
\end{gathered}
\end{equation}

Inequalities \eqref{eq3.18} and \eqref{eq3.19} imply the desirable boundedness. Passing from
$(y_m)_{m\in\mathbb{N}}$ to a suitable subsequence of $(y_m)_{m\in\mathbb{N}}$ it can be assumed
that
\begin{equation}\label{eq3.20}
\underset{m\to\infty}{\lim}y_m=b\quad\text{and}\quad
b\notin\{0,+\infty,-\infty\}.
\end{equation}
Moreover, using the conditions
$(x^i_m)_{m\in\mathbb{N}}\in\tilde{X}_p$, $i=1,\ldots,n$ and passing
to a subsequence again we can assume $\lim_{m\to\infty}\max\{d(x^1_m,p),\ldots,d(x^n_m,p)\}=0$ and
$\max\{d(x^1_m,p),\ldots,d(x^n_m,p)\}>0$ for $m\in\mathbb{N}$.

Thus the sequence $(r_m)_{m\in\mathbb{N}}$ with
$r_m=\max\{d(x^1_m,p),\ldots,d(x^n_m,p)\}$ can be selected as
normalizing. Using the obvious inequalities
$$
\frac{d(x^i_m,x^j_m)}{r_m}\leqslant 2\qquad\text{and}\qquad
\frac{d(x^i_m,p)}{r_m}\leqslant 1
$$
and passing to a subsequence again we can assume that the sequences
$\tilde{p},\tilde{x}_1=(x^1_m)_{m\in\mathbb{N}},\ldots,\tilde{x}_n=(x^n_m)_{m\in\mathbb{N}}$
are mutually stable. Let $\tilde{X}_{p,\tilde{r}}$ be a maximal
self-stable family for which
$\tilde{x}_i\in\tilde{X}_{p,\tilde{r}}$, $i=1,\ldots,n$ and
$\Omega^X_{p,\tilde{r}}$ be the corresponding pretangent space. Let
$\mathfrak{U}$ be a nontrivial ultrafilter on $\mathbb{N}$. Denote
by $\alpha_i$ the image of subsequence
$\tilde{x}_i=(x^i_m)_{m\in\mathbb{N}}$ under the projection of
$\tilde{X}_{p,\tilde{r}}$ on $\Omega^X_{p,\tilde{r}}$,
$\alpha_i=\pi(\tilde{x}_i)$. Now using properties $(i_2)$--$(i_3)$
and equality \eqref{eq3.20} we obtain
$$
\begin{gathered}
b=\mathfrak{U}-\lim\frac{\tau(Im(x^1_m,\ldots,x^n_m))-\diam_d(Im(x^1_m,\ldots,x^n_m))}
{r_m}\\=
\left(\mathfrak{U}-\lim\frac{\tau(Im(x^1_m,\ldots,x^n_m))}{r_m}\right)-
\left(\mathfrak{U}-\lim\frac{\diam_d(Im(x^1_m,\ldots,x^n_m))}{r_m}\right) \\
=\mathcal{X}_{\tau}(\{\alpha_1,\ldots,\alpha_n\})-
\diam_{\rho}(\{\alpha_1,\ldots,\alpha_n\}).
\end{gathered}
$$
Since $b\neq 0$ it implies the relation
$$
\mathcal{X}_{\tau}(\{\alpha_1,\ldots,\alpha_n\})\neq
\diam_{\rho}(\{\alpha_1,\ldots,\alpha_n\}),
$$
contrary to $(i)$.

The implication $(i)\Rightarrow(ii)$ follows.
\end{proof}

Theorem \ref{t3.10} and some known results about pretangent spaces
allow, in some cases, to get the relatively simple answer to the
question about infinitesimal structure of extended metrics
$\tau:\mathcal{F}(X)\to\mathbb{R}$ for which the corresponding
extended metrics $\mathcal{X}_{\tau}$ on pretangent spaces are
generated by metrics with some special properties.

Recall that a metric space $(X,d)$ is called {\it ultrametric} if
the inequality
$$
d(x,y) \leqslant d(x,z) \vee d(z,y)
$$
holds for all $x,y,z\in X$. Here and in the sequel we set $p\vee
q=\max\{p,q\}$ and $p\wedge q\ = \min\{p,q\}$ for all
$p,q\in\mathbb{R}$.

Let $(X,d)$ be a metric spaces with a marked point $p$. Let us
define a function $F_d:X^3\to\mathbb{R}$ as
$$
F_d(x,y):=\begin{cases} \frac{d(x,y)(d(x,p)\wedge
d(y,p))}{(d(x,p)\vee d(y,p))^2} \qquad\text{if}\qquad
(x,y)\neq(p,p)\\
0\qquad\qquad\qquad\qquad\text{ if}\qquad (x,y)=(p,p)
\end{cases}
$$
and a function $\Phi_d:X^3\to\mathbb{R}$ as
$\Phi_d(x,y,z):=F_d(x,y)\vee F_d(x,z)\vee F_d(y,z)$ for every
$(x,y,z)\in X^3$. For convenience we introduce the notations:
$d_1(x,y,z)$ is length of greatest side of the triangle with the
sides $d(x,y)$, $d(x,z)$ and $d(y,z)$ and $d_2(x,y,z)$ is length of
greatest of the two remained sides of this triangle.

\begin{lemma}\cite{DD}\label{l3.13}
Let $(X,d)$ be a metric space with a marked point $p$. All
pretangent spaces  $\Omega_{p,\tilde{r}}^X$ are ultrametric if and
only if
\begin{equation}\label{eq3.21}
\underset{x,y,z\rightarrow p}{\lim}
\Phi_d(x,y,z)\left(\frac{d_1(x,y,z)}{d_2(x,y,z)}-1\right)=0,
\end{equation}
where $\frac{d_1(x,y,z)}{d_2(x,y,z)}:=1$ for $d_2(x,y,z)=0$.
\end{lemma}

Using Theorem \ref{t3.10} and Lemma \ref{l3.13} we get
\begin{corollary}\label{c3.14}
Let $X\neq\varnothing$, $p\in X$ and let
$\tau:\mathcal{F}(X)\to\mathbb{R}$ be an extended by Balk metric.
The following statements are equivalent.
\begin{itemize}
\item[$(i)$] All extended metrics
$\mathcal{X}_{\tau}:\mathcal{F}(\Omega^X_{p,\tilde{r}})\to\mathbb{R}$
are generated by ultrametrics.

\item[$(ii)$] The extended metric $\tau:\mathcal{F}(X)\to\mathbb{R}$ is generated by the metric
$\tau^2$ at the point $p$ and the equality
$$
\underset{x,y,z\to p}{\lim} \Phi_{\tau^2}(x,y,z)
\left(\frac{\tau^2_1(x,y,z)}{\tau^2_2(x,y,z)}-1\right)=0,
$$
holds with $\frac{\tau^2_1(x,y,z)}{\tau^2_2(x,y,z)}:=1$ for
$x=y=z=p$.
\end{itemize}
\end{corollary}

\begin{remark}\label{r3.15}
An extended metric
$\mathcal{X}_{\tau}:\mathcal{F}(\Omega^X_{p,\tilde{r}})\to\mathbb{R}$
is generated by an ultrametric if and only if the inequality
$$
\mathcal{X}_{\tau}(A\cup B) \leqslant \mathcal{X}_{\tau}(A\cup
C)\vee \mathcal{X}_{\tau}(B\cup C)
$$
holds for all $A,B,C\in\mathcal{F}(\Omega^X_{p,\tilde{r}})$. (See
Theorem 2.1 in \cite{DDP}).
\end{remark}

{\bf O. Dovgoshey}

Institute of Applied Mathematics and Mechanics of NASU, R. Luxemburg
str. 74, Donetsk 83114, Ukraine

{\bf E-mail: } aleksdov@mail.ru

\bigskip

{\bf D. Dordovskyi}

Institute of Applied Mathematics and Mechanics of NASU, R. Luxemburg
str. 74, Donetsk 83114, Ukraine

{\bf E-mail: } dordovskydmitry@gmail.com

\end{document}